\documentclass[10pt]{article}
\usepackage[latin1]{inputenc}
\usepackage{epsfig}
\usepackage{color}
\usepackage[british,english]{babel}
\usepackage{amsthm}
\usepackage{amsmath}
\usepackage{amsfonts}
\usepackage{amssymb}
\usepackage{graphicx}
\def \noame{\noalign{\medskip}}
\newtheorem{corollary}{Corollary}[section]
\newtheorem{definition}[corollary]{Definition}

\newtheorem{lemma}[corollary]{Lemma}
\newtheorem{proposition}[corollary]{Proposition}
\newtheorem{remark}[corollary]{Remark}
\newtheorem{theorem}[corollary]{Theorem}
\newfont{\sBlackboard}{msbm10 scaled 900}

\newcommand{\mylabel}[1]{\label{#1}
            \ifx\undefined\stillediting
            \else \fbox{$#1$}\fi }
\newcommand{\BE}{\begin{equation}}

\newcommand{\EEQ}{\end{equation}}
\newcommand{\rfb}[1]{\mbox{\rm
   (\ref{#1})}\ifx\undefined\stillediting\else:\fbox{$#1$}\fi}

\newfont{\Blackboard}{msbm10 scaled 1200}

\newfont{\roma}{cmr10 scaled 1200}

\def\CC{\rm \hbox{C\kern-.56em\raise.4ex
         \hbox{$\scriptscriptstyle |$}\kern+0.5 em }}

\newcommand{\ep}{\varepsilon}

\def\n{|\kern -.05cm{|}\kern -.05cm{|}}

\newcommand{\mm}    {{\hbox{\hskip 0.5pt}}}

\newcommand{\bluff} {{\hbox{\raise 15pt \hbox{\mm}}}}
\usepackage{fancyhdr}

\makeatletter
\def\section{\@startsection {section}{1}{\z@}{-3.5ex plus -1ex minus
    -.2ex}{2.3ex plus .2ex}{\large\bf}}
\makeatother
\def\be{\begin{equation}}
\def\ee{\end{equation}}

\date{ }
\begin{document}
\thispagestyle{empty}
\title{\bf Theoretical derivation of Darcy's law for fluid flow \\
in thin porous media}
\maketitle
\vspace{-50pt}
\begin{center}
Francisco Javier SU\'AREZ-GRAU\footnote{Facultad de Matem\'aticas. Universidad de Sevilla. 41012-Sevilla (Spain) grau@us.es}
 \end{center}

\vskip20pt

 \renewcommand{\abstractname} {\bf Abstract}
\begin{abstract} 
In this paper we study  stationary incompressible Newtonian fluid flow in a thin porous media. The media under consideration is a bounded perforated $3D$ domain confined between two parallel plates. The description of the domain  includes two small parameters:  $\ep$ representing the distance between pates and $a_\ep$ connected to the microstructure of the domain such that $a_\ep\ll \ep$.  We consider the classical setting of perforated media, i.e. $a_\ep$-periodically distributed solid (not connected) obstacles of  size $a_\ep$.  The goal of this paper is to introduce a version of the unfolding method, depending on both parameters $\ep$ and $a_\ep$, and then to
derive the corresponding  $2D$ Darcy's law. \end{abstract}
\bigskip\noindent
 {\small \bf AMS classification numbers:}  76A20, 76M50, 35B27, 35Q30. 

\section {Introduction}\label{S1}
The problem of Stokes fluid flows in a periodically perforated domains with obstacles of the same
size as the period has been widely treated in the literature. As is well known, such kind of flows are generally modelled by Darcy's law, see Darcy \cite{Darcy}. From the mathematical point of view, the transition between Stokes equations to Darcy's law was formally obtained in Sanchez-Palencia \cite{Sanchez}  and rigorously  in Tartar  \cite{Tartar}. Since then,  several approaches to derive Darcy's law have been used by classical authors in homogenization theory, such as Allaire \cite{Allaire0}, Hornung \cite{Hornung} and Lions \cite{Lions} among others. The goal of this paper is to generalize classical results of perforated media to the case of thin porous media, which  by definition includes two small parameters: one called $\ep$ is connected to the  fluid  film thickness and the other denoted by $a_\ep$ to the microstructure representing the size of the obstacles and the interspatial distance between them. 

The case of thin porous media has been recently considered taking as microstructure a periodic array of vertical cylinders  of size and period $a_\ep$ confined between two parallel plates with distance $\ep$ in Fabricius {\it et al.} \cite{Fabricius}.  Thus, depending on the relation between parameters $\ep$ and $a_\ep$, different regimes are obtained:  proportionally thin porous media ($a_\ep\approx \ep$),  homogeneously thin porous media ($a_\ep\ll \ep$) and very thin porous medium ($a_\ep\gg \ep$). For each case, a permeability tensor is obtained by solving local problems. In the critical case, the local problems are 3D, while they are 2D in the
other cases, which is a considerable simplification. This result is proved in Fabricius {\it et al.}  \cite{Fabricius} by using the multiscale expansion method, which is a formal but powerful tool to analyze homogenization problems, and later rigorously developed in Anguiano and Su\'arez-Grau \cite{Anguiano_MJOM} by using an adaptation of the periodic unfolding method, see Arbogast {\it et al.} \cite{arbogast} and Cioranescu {\it et al.} \cite{Ciora, Ciora2}, which is introduced in Anguiano and Su\'arez-Grau \cite{Anguiano1}.

The goal of this paper is to consider the classical microstructure of perforated domains, i.e. consider  $a_\ep$-periodically distributed solid (not connected) obstacles of  the same size  and then, to study the influence of the thickness of the domain $\ep$ to derive the corresponding Darcy's law.  In this case, the restriction $a_\ep\ll \ep$ has to be naturally imposed. To our knowledge, this problem has only been considered  in the case of modelling  a thin film passing a thin porous media by Bayada \cite{Bayada_coupling}, but only the 2D case is considered. In this paper,  we consider the 3D case and prove the convergence of the homogenization process when $\ep$ and $a_\ep$ go to zero. To to this,  the microgeometry  of the thin porous media requires a version of the unfolding method depending on both parameters $\ep$ and $a_\ep$, which can be applied to other problems and also in the 2D case.  As a result, we rigorously derive the corresponding 2D Darcy's law, which is different compared to that  obtained in the case of homogeneously thin porous media obtained in Fabricius {\it et al.}  \cite{Fabricius}. 

The structure of the paper is as follows. In Section \ref{sec:domain} we introduce the domain and some useful notation is given in Section \ref{sec:notation}. In Section \ref{sec:statement}, we describe the statement of the problem and give the main result (Theorem \ref{mainthm}), whose proof is provided in Section \ref{sec:proofs}. 

\section{Definition of the domain} \label{sec:domain} Let $\omega$ be a smooth, bounded and connected set in $\mathbb{R}^2$. We consider a positive parameter $\ep$ which describes the height of the domain, i.e. we define $Q_\ep=\omega\times (0,\ep)$ as the thin domain without microstructure. 
To describe the microstructure of the domain $Q_\ep$, we consider the positive and small parameter $a_\ep$ such that 
\begin{equation}\label{parameters}
\lim_{\ep\to 0} {a_\ep\over \ep}=0.
\end{equation}
We denote $Y=(-1/2,1/2)^3$ the unitary cube in $\mathbb{R}^3$ as the reference cell and  $T$ an open connected subset of $Y$ with a smooth boundary  $\partial T$ such that $\overline T\subset Y$. We denote $Y_f=Y\setminus \overline T$. Thus, for $k\in\mathbb{Z}^3$, each cell $Y_{k,a_\varepsilon}=a_\varepsilon k+a_\varepsilon Y$ is similar to the unit cell $Y$ rescaled to size $a_\varepsilon$ and $T_{k,a_\varepsilon}=a_\varepsilon k+a_\varepsilon T$ is similar to $T$ rescaled to size $a_\varepsilon$. We denote $Y_{f_k,a_\varepsilon}=Y_{k,a_\varepsilon}\setminus \overline T_{k,a_\varepsilon}$ (see Figure \ref{fig:cell}).
\begin{figure}[h!]
\begin{center}
\includegraphics[width=4cm]{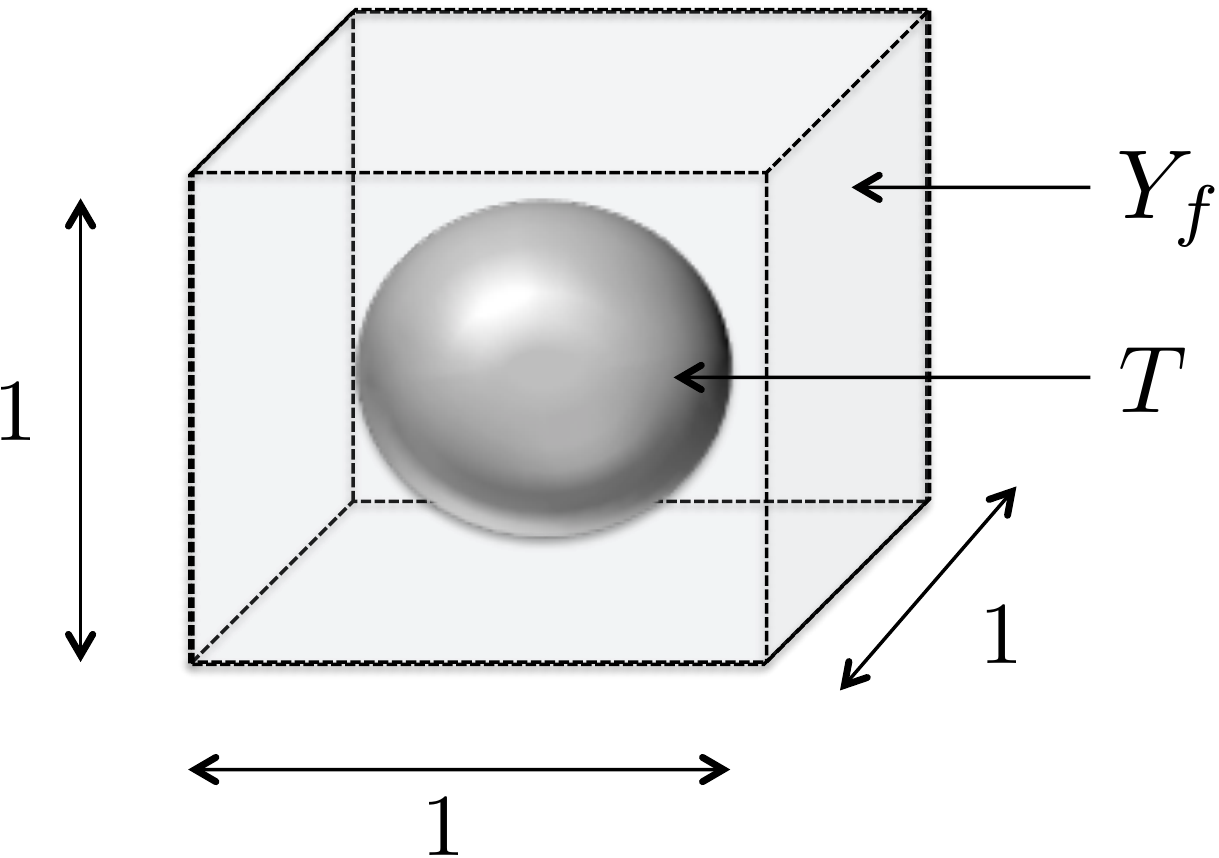}
\hspace{2.5cm}
\raisebox{.5\height}{\includegraphics[width=3cm]{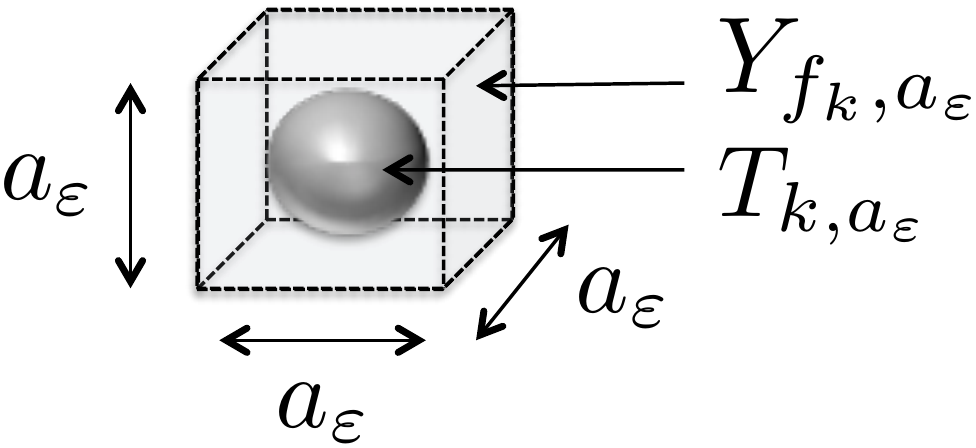}}
\end{center}
\vspace{-0.4cm}
\caption{View of the reference cell  $Y$ (left) and the rescaled cell $Y_{k,a_\ep}$ (right).}
\label{fig:cell}
\end{figure}
We denote by $\tau(\overline T_{k,a_\ep})$ the set of all translated images of $\overline T_{k,a_\ep}$. The set $\tau(\overline T_{k,a_\ep})$ represents the obstacles in $\mathbb{R}^3$. The thin porous media $\Omega_\ep$ is defined by  (see Figure \ref{fig:domain})
\begin{equation}\label{Omegaep}
\Omega_\ep=Q_\ep\setminus \bigcup_{k\in \mathcal{K}_\ep}\overline T_{k,a_\varepsilon},
\end{equation} where   $\mathcal{K}_\ep:=\left\{k\in\mathbb{Z}^N\,:\, Y_{k,a_\varepsilon}\cap Q_\ep\neq \emptyset\right\}$. By construction,  $\Omega_\varepsilon$ is a periodically perforated domain with obstacles of the same size as the period.
We make the assumption that the obstacles $\tau(\bar T_{k,a_\ep})$ do no intersect the boundary $\partial Q_\ep$. We denote by $T_{\varepsilon}$ the set of all the obstacles contained in $\Omega_\ep$. Then,   $T_\ep$ is a finite union of obstacles, i.e. 
 $$T_{\varepsilon}=\bigcup_{k\in \mathcal{K}_\ep}\overline T_{k,a_\ep}.$$
 \begin{figure}[h!]
\begin{center}
\includegraphics[width=7cm]{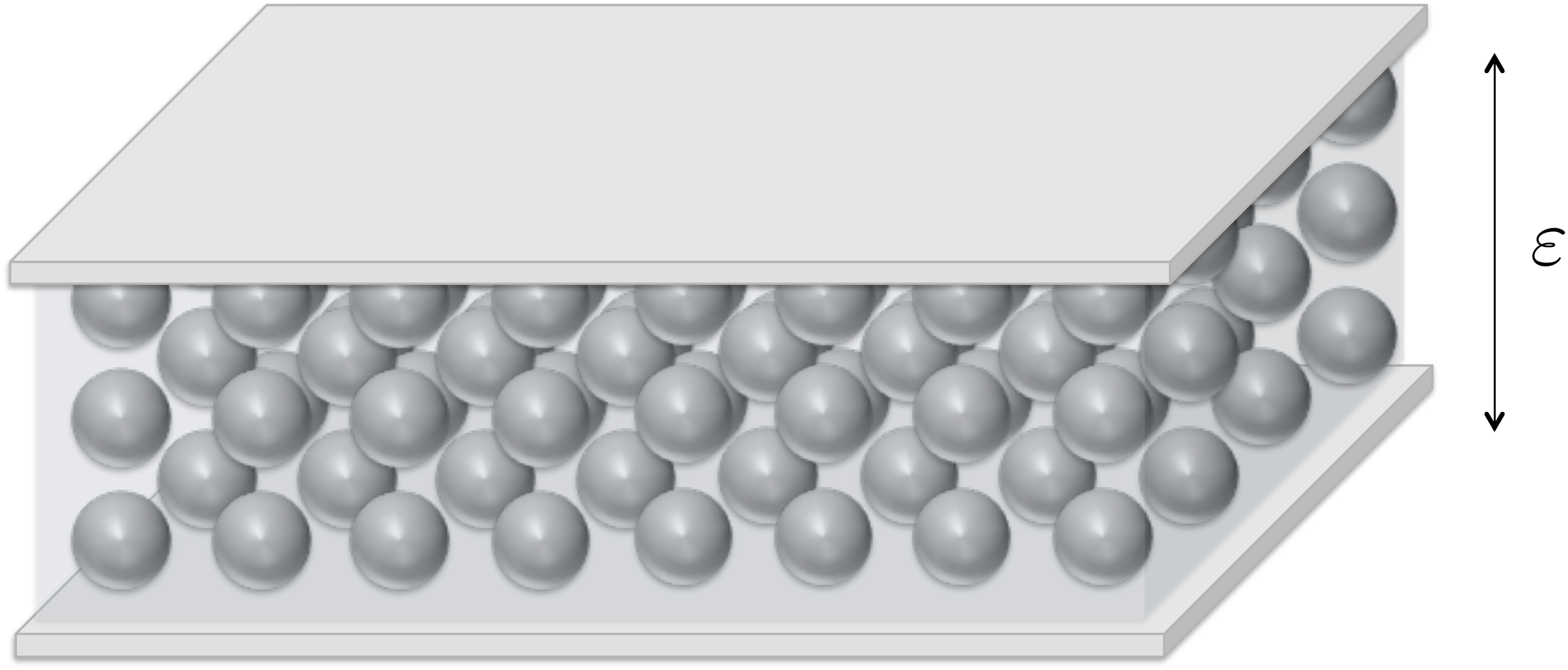}
\hspace{2cm}
\includegraphics[width=7cm]{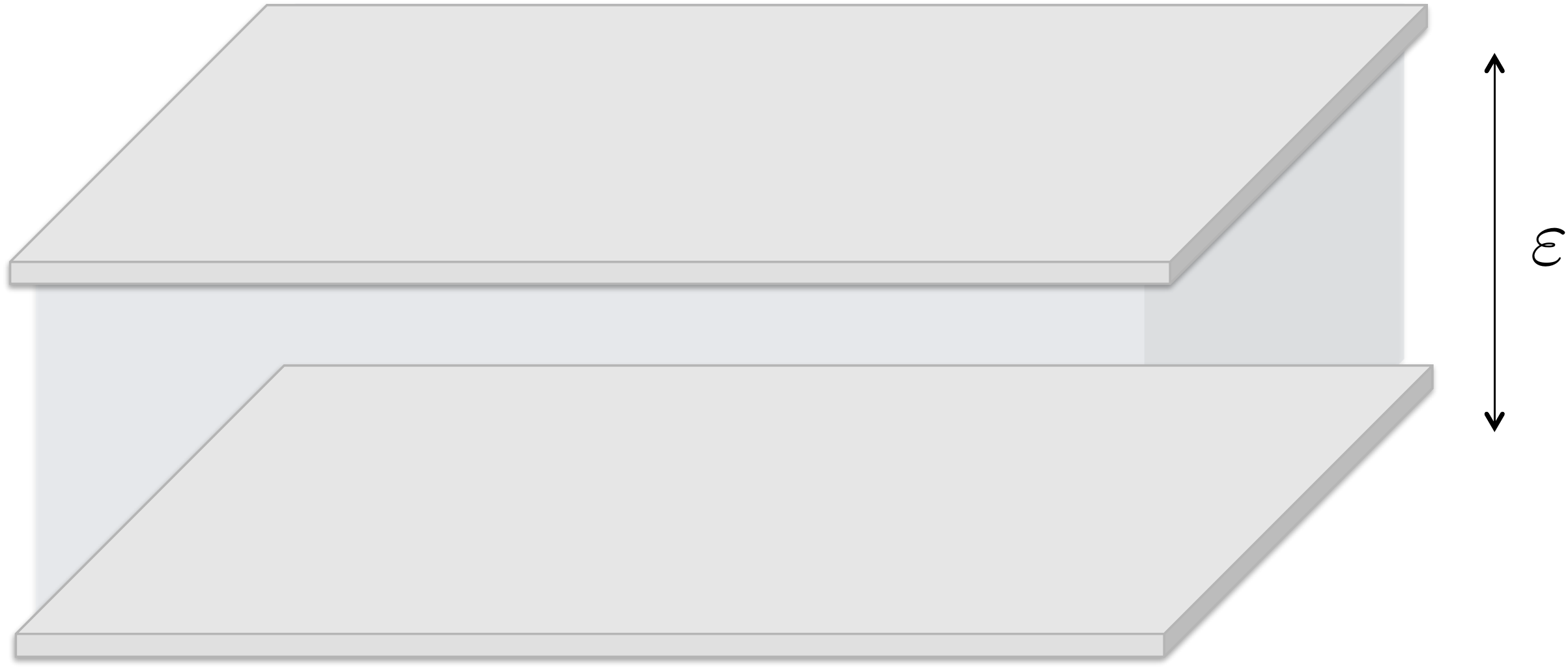}
\end{center}
\caption{View of the domain $\Omega_{\varepsilon}$ (left) and $Q_\ep$ (right).}
\label{fig:domain}
\end{figure}

\noindent As usual when we deal with thin domains, we will use the dilatation in the variable $x_3$ given by
\begin{equation}\label{dilatacion}
z_3={x_3\over \ep}.
\end{equation}
Then, we define the rescaled porous media  $\widetilde \Omega_\ep$ by (see Figure \ref{fig:domain2})
\begin{equation}\label{Omega_tilde}
\widetilde\Omega_\ep=\left\{(x',z_3)\in \mathbb{R}^3\,:\, (x',\ep z_3)\in \Omega_\ep\right\}.
\end{equation}
We also introduce the rescaled sets $\widetilde Y_{k,a_\ep}$ by  (see Figure \ref{fig:domain2})
$$\widetilde Y_{{k},a_\ep}=\left\{(x',z_3)\in\mathbb{R}^3\,:\, (x',\ep z_3)\in Y_{k,a_\varepsilon}\right\},$$
and, in the same way, we define the rescaled fluid part $\widetilde Y_{f_{k},a_\ep}$, the rescaled solid  part $\widetilde T_{k,a_\varepsilon}$ of $\widetilde Y_{{k},a_\ep}$ and the union of rescaled obstacles $\widetilde T_\ep$.
Finally, by $\Omega$ we denote the domain with fixed height without microstructure
$$
\Omega=\omega\times (0,1). 
$$
 \begin{figure}[h!]
\begin{center}
\includegraphics[width=7cm]{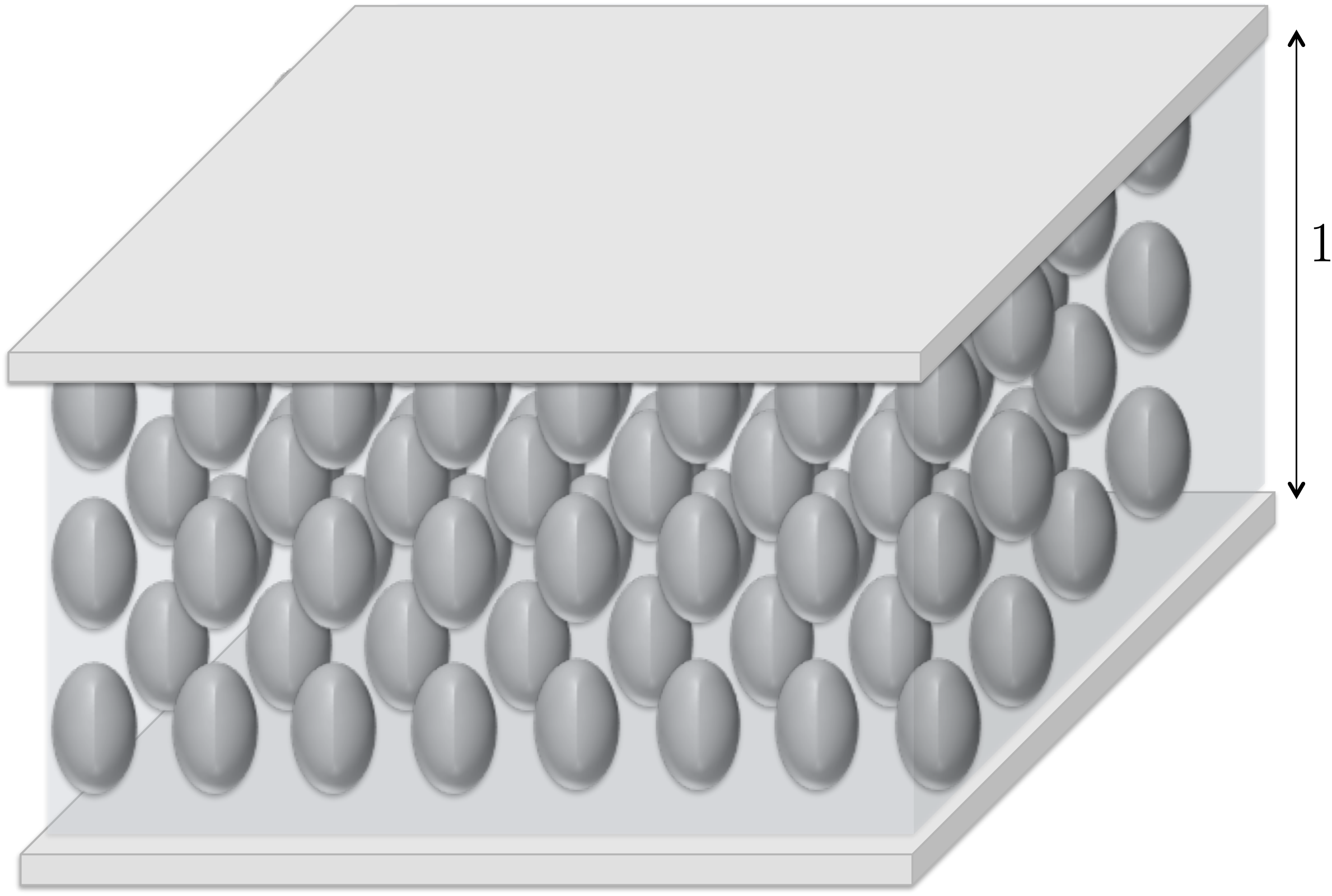}
\hspace{3cm}
\raisebox{.6\height}{\includegraphics[width=3cm]{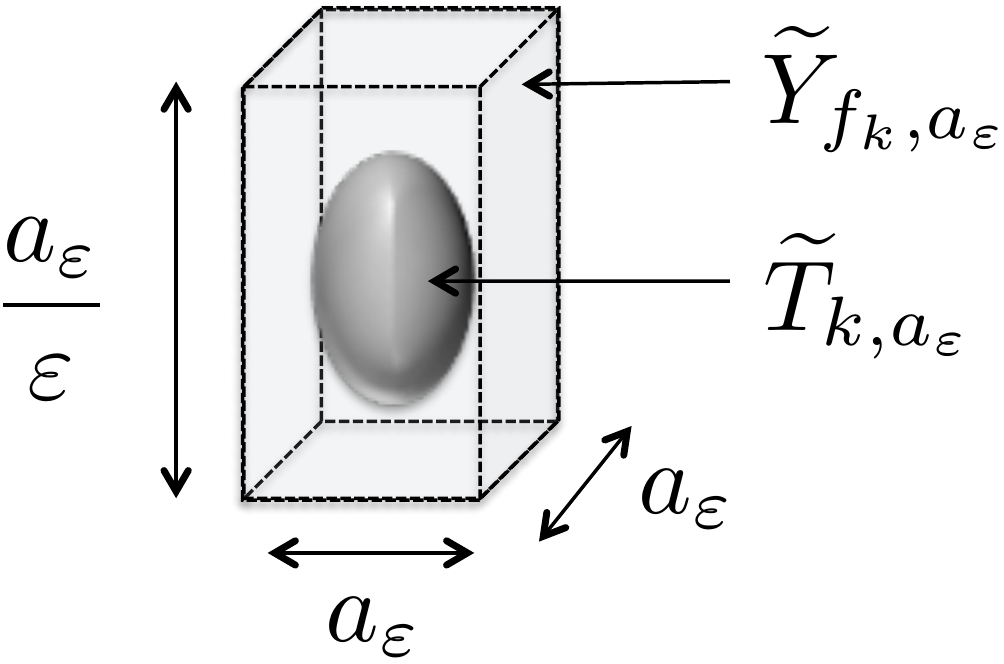}}
\end{center}
\caption{View of the rescaled domain $\widetilde \Omega_{\varepsilon}$ (left) and the rescaled cell $\widetilde Y_{k,a_\ep}$ (right).}
\label{fig:domain2}
\end{figure}

\section{Some notation}\label{sec:notation} Along this paper, the points $x\in\mathbb{R}^3$ will be decomposed as $x=(x',x_3)$ with $x'\in\mathbb{R}^2$, $x_3\in\mathbb{R}$. We also use the notation $x'$ to denote a generic vector of $\mathbb{R}^2$.

In order to apply the version of the unfolding method, we need the following notation: for $k\in\mathbb{Z}^3$, we define $\kappa:\mathbb{R}^3\to \mathbb{Z}^3$ by
\begin{equation}\label{kappa}
\kappa(x)=k\Longleftrightarrow x\in Y_{k,1}.
\end{equation}
Remark that $\kappa$ is well defined up to a set of zero measure in $\mathbb{R}^3$, which is given by $\cup_{k\in\mathbb{R}^3}\partial Y_{k,1}$.  Moreover, for every $\ep,\,a_\ep>0$, we have 
$$\kappa\left({x\over a_\ep}\right)=k\Longleftrightarrow x\in Y_{k,a_\ep}\quad\hbox{which is equivalent to }\quad \kappa\left({x'\over a_\ep},{\ep z_\ep\over a_\ep}\right)=k\Longleftrightarrow (x',z_3)\in \widetilde Y_{k,a_\ep}.$$
For a vectorial function $\varphi=(\varphi',\varphi_3)$ and a scalar function $\psi$, we introduce the operators $D_{\ep}$, $\nabla_{ \ep}$ and ${\rm div}_{ \ep}$ by 
\begin{eqnarray}
&\displaystyle(D_{ \ep}\varphi)_{ij}=\partial_{x_j}\varphi_i\hbox{ for }i=1,2,3,\ j=1,2,\quad (D_{ \ep}\varphi)_{i,3}=\ep^{-1}\partial_{y_3}\varphi_i\hbox{ for }i=1,2,3,\nonumber&\\
\noame
&\displaystyle\nabla_{ \ep}\psi=(\nabla_{x'}\psi,\ep^{-1}\partial_{y_3}\psi)^t,\quad {\rm div}_{ \ep}\varphi={\rm div}_{x'}\varphi'+\ep^{-1}\partial_{y_3}\varphi_3.  &\nonumber 
\end{eqnarray}
We denote by $L^2_{\rm per}(Y)$, $H^1_{\rm per}(Y)$, the functional spaces
$$\begin{array}{l}
\displaystyle L^2_{\rm per}(Y)=\Big\{\varphi\in L^2_{\rm loc}(\mathbb{R}^3)\,:\, \int_Y|\varphi|^2<+\infty,\quad v(y+k)=v(y)\quad\forall\,k\in\mathbb{Z}^3,\ \hbox{a.e. }y\in Y\Big\},\\
\noame
\displaystyle
H^1_{\rm per}(Y)=\Big\{\varphi\in H^1_{\rm loc}(\mathbb{R}^3)\cap L^2_{\rm per}(Y)\,:\,\nabla \varphi\in L^2_{\rm per}(Y)^3\Big\}.
\end{array}$$
We denote by $:$ the full contraction of two matrices, i.e. for $A=(a_{ij})_{1\leq i,j\leq 3}$ and $B=(a_{ij})_{1\leq i,j\leq 3}$, we have $A:B=\sum_{i,j=1}^3a_{ij}b_{ij}$.

The canonical basis in $\mathbb{R}^3$ is denoted by $\{e',e_3\}$ where $e'=\{e_1,e_2\}$.

Finally, we denote by $O_\ep$ a generic real sequence, which tends to zero with $\ep$ and can change from line to line, and by $C$ a generic positive constant which also can change from line to line.

\section{Statement of the problem and main result}\label{sec:statement}  We consider the following Stokes system in $\Omega_\ep$, with homogeneous boundary conditions in the boundary of the obstacles  $\partial T_{\ep}$ and  the exterior boundary $\partial Q_\ep$,
\begin{equation}\label{system_1_dimension_1}
\left\{\begin{array}{rl}
-\nu\,{\rm div} (D  u_\ep)+\nabla p_\ep=f_\ep &\quad\hbox{in}\quad\Omega_\ep,\\
\noame
{\rm div}\,u_\ep=0&\quad\hbox{in}\quad\Omega_\ep,\\
\noame
 u_\varepsilon=0 & \hbox{ on }\partial T_\varepsilon \cup\partial Q_\ep.
\end{array}\right.
\end{equation}
In the above system, velocity filed $u_\ep$ and (scalar) pressure $p_\ep$  are unknown, while $f_\ep$ represents the body force and the viscosity is denoted by $\nu$.

To study the asymptotic behavior of the solution  we will use the equivalent variational formulation of (\ref{system_1_dimension_1}) which is the following one: find $u_\ep \in H^1_0(\Omega_\ep)^3$ and $p_\ep\in L^2_0(\Omega_\ep)$ such that
\begin{equation}\label{form_var_1}
 \begin{array}{l}
\displaystyle \nu\int_{\Omega_\ep}D u_\ep:D\varphi\,dx-\int_{\Omega_\ep}p_\ep\,{\rm div}\,\varphi\,dx=\int_{\Omega_\ep}f_\ep\cdot\varphi\,dx,
\quad\forall\,\varphi  \in H^1_0(\Omega_\ep)^3.
\end{array} 
\end{equation}
Assuming  external body force $f_\ep\in L^2(\Omega_\ep)^3$, it is well known that (\ref{system_1_dimension_1})  has a unique weak solution $(u_\ep, p_\ep)\in H^1_0(\Omega_\ep)^3\times L^2_0(\Omega_\ep)$, see e.g \cite{Temam}.

Our aim is to describe the asymptotic behavior of the velocity $u_\ep$ and the pressure
$p_\ep$ of the fluid as $\ep$ tends to zero and identify homogenized models coupling the effects of the thickness and its microgeometry. For this purpose, as usual when we deal with thin domains, we use the dilatation (\ref{dilatacion}) in order to have the functions defined in the open set with fixed height $\widetilde\Omega_\ep$ defined in (\ref{Omega_tilde}). Namely, we define  $\tilde u_\ep \in H^1_0(\widetilde\Omega_\ep)^3$ and $\tilde p_\ep\in L^2_0(\widetilde\Omega_\ep)$ by 
\begin{equation}\label{unk_dilat}
\tilde u_\ep(x',z_3)=u_\ep(x', \ep z_3),\quad  \tilde p_\ep(x',z_3)=p_\ep(x', \ep z_3),\quad\hbox{a.e. }(x',z_3)\in \widetilde\Omega_\ep\,.
\end{equation}
Using the transformation (\ref{dilatacion}) the rescaled system (\ref{system_1_dimension_1}) can be rewritten as
\begin{equation}\label{system_2}
\left\{
\begin{array}{rl}
-\nu\,{\rm div}_{\ep}(D_{\ep} \tilde u_\ep)+\nabla_{\ep} \tilde p_\ep= \tilde f_\ep&\quad\hbox{in}\quad\widetilde\Omega_\ep,\\
\noame
{\rm div}_{\ep}\tilde u_\ep=0&\quad\hbox{in}\quad\widetilde \Omega_\ep,\\
\noame
\tilde u_\ep=0&\quad \hbox{on}\quad\partial \widetilde T_\ep\cup \widetilde\Omega\,,
\end{array}\right.
\end{equation}
where $\tilde f_\ep$ is defined similarly as in (\ref{unk_dilat}).  Moreover, taking in (\ref{form_var_1}) as test function $\tilde\varphi(x',x_3/\ep)$ with $\tilde\varphi\in H^1_0(\widetilde\Omega_\ep)^3$, applying the change of variables (\ref{dilatacion}), the variational formulation for the rescaled system (\ref{system_2}) is then to find $\tilde u_\ep \in H^1_0(\widetilde \Omega_\ep)^3$ and $\tilde p_\ep\in L^2_0(\widetilde \Omega_\ep)$ such that
\begin{equation}\label{form_var_2}
 \begin{array}{l}
\displaystyle \nu \int_{\widetilde \Omega_\ep}D_{\ep} \tilde u_\ep:D_{\ep}\tilde\varphi\,dx'dy_3-\int_{\widetilde \Omega_\ep}\tilde p_\ep\,{\rm div}_{\ep}\tilde\varphi\,dx'dy_3=\int_{\widetilde \Omega_\ep}\tilde f_\ep\cdot\tilde\varphi\,dx'dy_3\,,\quad \forall\,\varphi \in H^1_0(\widetilde \Omega_\ep)^3.\end{array} 
\end{equation}
We point out that due to the thickness of the domain, it is usual to assume that the vertical component of the external force can be neglected. Thus, let us apply to the fluid an external body force defined by 
\begin{equation}\label{fep}
f_\ep=(f'(x'),0)\quad \hbox{  with  }f'\in L^2(\omega)^2.
\end{equation}
Our goal then becomes describing the asymptotic behavior of this new sequences $\tilde u_\ep$  and $\tilde p_\ep$ when $\ep$ tends to zero.  However, the sequence of solutions $(\tilde u_\ep, \tilde p_\ep)\in H^1_0(\widetilde\Omega_\ep)^3\times L^2_0(\widetilde\Omega_\ep)$ is not defined in a fixed domain independent of $\ep$ and $a_\ep$ but rather in a varying set $\widetilde\Omega_\ep$. Thus, in order to pass to the limit if $\ep$ tends to zero, convergences in fixed Sovolev spaces (defined in $\Omega$) are used which requires first that $(\tilde u_\ep, \tilde p_\ep)$ be extended to the whole domain $\Omega$. Then, by definition, an extension $(\tilde U_\ep, \tilde P_\ep)\in H^1_0(\Omega)^3\times L^2_0(\Omega)$ of $(\tilde u_\ep, \tilde p_\ep)$ is defined on $\Omega$ and coincides with $(\tilde u_\ep, \tilde p_\ep)$ on $\widetilde \Omega_\ep$. 

Our main result  is referred to the asymptotic behavior of the solution of (\ref{system_2}) and is given by the following theorem.

\begin{theorem}\label{mainthm}Let $(\tilde u_\ep,\tilde p_\ep)$ be the unique solution of problem (\ref{system_2}). Then the sequence of the extension $(a_\ep^{-2}\tilde U_\ep, \tilde P_\ep)$ converges weakly to $(u, p)$ in $L^2(\Omega)^3\times L^2(\Omega)$ with $u_3=0$. Moreover, defining $U(x')=\int_0^1u(x',z_3)\,dz_3$, it holds 
\begin{equation}\label{Darcy_velocity}
U'(x')=K\left(f'(x')-\nabla_{x'}p(x')\right),\quad U_3(x')=0\quad \hbox{in }\omega,
\end{equation}
and $p\in H^1(\omega)\cap L^2_0(\omega)$ is the unique solution of the $2D$ Darcy equation 
\begin{equation}\label{Darcy}
\left\{\begin{array}{rl}
\displaystyle
{\rm div}_{x'} \Big(K\left(f'(x')-\nabla_{x'}p(x')\right)\Big)=0&\hbox{ in }\omega,\\
\noame
\displaystyle
\Big(K\left(f'(x')-\nabla_{x'}p(x')\right)\Big)\cdot n=0&\hbox{ on }\partial\omega.
\end{array}\right.
\end{equation}
Here $K\in\mathbb{R}^{2\times 2}$ is a definite positive matrix with coefficients
$$
K_{ij}=\nu\int_{Y_f}Dw^i(y):Dw^j(y)\,dy,\quad i,j=1,2,
$$
where $w^i$, $i=1,2$, is  the unique solution of the local Stokes problem 
\begin{equation}\label{Local_problems}
\left\{\begin{array}{rl}
-\nu\Delta_y w^i+ \nabla_y \pi^i=e_i & \hbox{ in }Y_f,\\
\noame
{\rm div}_y w^i=0& \hbox{ in }Y_f,\\
\noame
w^i=0& \hbox{ in }T,\\
\noame
w^i\in H^1_{{\rm per}}(Y_f)^3,& \pi^i\in L^2_{\rm per}(Y_f)/\mathbb{R}.
\end{array}\right.
\end{equation}
\end{theorem}

%
%
%
%
\begin{remark}As stated in the Introduction, the derivation of Darcy's law flows in thin porous media has been studied in \cite{Anguiano_MJOM, Fabricius} where the microstructure is a periodic array of vertical cylinders of size and period $a_\ep$ and thickness $\ep$. As a result,  three different regimes depending on the relation of parameters $\ep$ and $a_\ep$ are obtained.  The relation considered in this paper, $a_\ep\ll \ep$, can be framed in the regime called {\it homogeneously thin porous media}. Although the current study and the previous ones allow to derive Darcy's law (\ref{Darcy}) with permeability tensor $K$ obtained by solving local problems, the main difference between them lies in the local problems (\ref{Local_problems}), because the particular structure of the cylinders in the case of homogeneously thin porous media allows local problems to be reduced to problems defined in a 2D reference cell while they are defined in 3D in this case.
\end{remark}

\section{Proofs}\label{sec:proofs}
In this section we provide the proof of the main result (Theorem \ref{mainthm}). To to this,  we define of  the extension of the solution and we establish  some {\it a priori} estimates  in Subsection \ref{sec:estimates}. In Subsection \ref{sec:unfolding} we  introduce the version of the 
unfolding method depending on both parameters $\ep$ and $a_\ep$.   A compactness result, which is the main key when we will pass to the limit later, is addressed
in Subsection \ref{sec:compactness}. Finally,  the proof of the Theorem \ref{mainthm}  is given in Subsection \ref{sec:mainthm}.

\subsection{{\it A priori} estimates}\label{sec:estimates}
In this subsection, we establish sharp {\it a priori} estimates of the dilated solution in $\widetilde \Omega_\ep$. To do this, we first need  the  Poincar\'e inequality in thin porous media (for the classical version see  \cite{Tartar}).
\begin{lemma} \label{Lemma_Poincare}There exists a positive constant $C$, independent of $\ep$, such that
\begin{equation}\label{Poincare1}
\|\varphi\|_{L^2(\Omega_\ep)^3}\leq Ca_\ep\|D \varphi\|_{L^2(\Omega_\ep)^{3\times 3}},\quad \forall\,\varphi\in H^1_0(\Omega_\ep)^3.
\end{equation}
Moreover, considering $\tilde \varphi\in H^1_0(\widetilde\Omega_\ep)^3$ such that  $\tilde \varphi(x', z_3)=\varphi(x',\ep z_3)$, then it holds
\begin{equation}\label{Poincare}
\|\tilde \varphi\|_{L^2(\widetilde\Omega_\ep)^3}\leq Ca_\ep\|D_\ep \tilde \varphi\|_{L^2(\widetilde \Omega_\ep)^{3\times 3}}.
\end{equation}
\end{lemma}
\begin{proof} We observe that $\Omega_\ep$ can be divided in small cubes of lateral and vertical length $a_\ep$. We consider the periodic cell $Y_f$  and  have a Friedrichs inequality    
\begin{equation}\label{proof_gaffney}
\int_{Y_f}|\varphi|^2\,dz\leq C\int_{Y_f}|D\varphi|^2\,dz\,,
\end{equation}
for every $\varphi\in H^1(Y_f)^3$ such that $\varphi =0$ on $\partial T$, where where the constant $C$ depends only on $Y_f$. Then, for every $k\in \mathbb{R}^3$, by the change of variable
\begin{equation}\label{change}
\begin{array}{l}
\displaystyle k+z={x\over a_\ep},\quad dz={dx\over a_\ep^3},\quad  \partial_{z}=a_\ep\partial_{x},\end{array}
\end{equation}
we rescale (\ref{proof_gaffney}) from $Y_f$ to $Y_{f_k,a_\ep}$.  This yields that, for every function $\varphi(x)\in H^1(Y_{f_k,a_\ep})^3$, one has
$$
\int_{Y_{f_k,a_\ep}}|\varphi|^2\,dx\leq C\int_{Y_{f_k,a_\ep}}|D\varphi|^2\,dx\,,
$$
with the same constant $C$ as in (\ref{proof_gaffney}).  Summing previous inequality for every $k\in \mathcal{K}_\ep$ and gives (\ref{Poincare1}). 

In fact, we must consider separately the periods containing a portion of $\partial Q_\ep$, but they yield at a distance $O(a_\ep)$
of $\partial Q_\ep$, where $\varphi$ is zero, and then the corresponding inequality is immediately obtained.

Finally, by taking in (\ref{Poincare1}) the change of variables (\ref{dilatacion}), we get (\ref{Poincare}).
\end{proof}

We give {\it a priori} estimates for velocity  $\tilde u_\ep$ in $\widetilde \Omega_\ep$.
\begin{lemma} \label{Estimates_lemma} There exists a positive constant $C$, independent of $\ep$, such that 
\begin{equation}\label{estimates_u_tilde}
\|\tilde u_\ep\|_{L^2(\widetilde \Omega_\ep)^{3}}\leq Ca_\ep^2 ,\quad \|D_\ep \tilde u_\ep\|_{L^2(\widetilde \Omega_\ep)^{3\times 3}}\leq Ca_\ep\,.
\end{equation}
\end{lemma}
\begin{proof} Considering $\tilde u_\ep$ as test function in (\ref{form_var_2}) and apply Cauchy-Schwarz's inequality to obtain
$$\begin{array}{l}
\displaystyle \nu\|D_\ep \tilde u_\ep\|_{L^2(\widetilde \Omega_\ep)^{3\times 3}}^2
 \displaystyle \leq \|\tilde f_\ep\|_{L^2(\widetilde \Omega_\ep)^3}
\|\tilde u_\ep\|_{L^2(\widetilde \Omega_\ep)^3}.
\end{array}$$
Taking into account the assumption on $f_\ep$, which in particular does not depend on $x_3$, we have
$$\begin{array}{l}
\displaystyle \|D_\ep \tilde u_\ep\|_{L^2(\widetilde \Omega_\ep)^{3\times 3}}^2
 \displaystyle \leq  C \|\tilde u_\ep\|_{L^2(\widetilde \Omega_\ep)^3},
\end{array}$$
and from Poincar\'e's inequality (\ref{Poincare}), we get
$$\begin{array}{l}
\displaystyle \|D_\ep \tilde u_\ep\|_{L^2(\widetilde\Omega_\ep)^{3\times 3}}^2
 \displaystyle \leq C a_\ep\|D_\ep \tilde u_\ep\|_{L^2(\widetilde\Omega_\ep)^{3\times 3}},
\end{array}$$
which proves the second estimate in (\ref{estimates_u_tilde}). This and again Poincar\'e's inequality (\ref{Poincare}) give the first one.\end{proof}

\paragraph{The extension of $(\tilde u_\ep, \tilde p_\ep)$ to the whole domain $\Omega$.}
We extend the velocity $\tilde u_\ep$ by zero in $\Omega\setminus\widetilde \Omega_\ep$ (this is compatible with the homogeneous boundary condition on $\partial \Omega_\ep$), and   denote the extension by $\tilde U_\ep$. Obviously, estimates given in Lemma \ref{Estimates_lemma} remain valid and the extension $\tilde U_\ep$ is divergence free too.

 In order to extend the pressure $\tilde p_\ep$ to the whole domain $\Omega$,  we recall an important result from \cite{Tartar} which  is generalizated to thin domains  in \cite{Anguiano1}  and is concerned with the extension of the pressure $p_\ep$ to the whole domain $Q_\ep$.  Thus,  we first use a restriction operator $R^\ep$ from $H^1_0(Q_\ep)^3$ into $H^1_0(\Omega_\ep)^3$ which is introduced in  Lemma 4.5 in  \cite{Anguiano1} as $R^\ep_2$, next  we extend the gradient of the pressure by duality in $H^{-1}(Q_\ep)^3$ and finally by means of the dilatation we extend $\tilde p_\ep$ to $\Omega$. 
 
\begin{lemma} \label{restriction_operator}
There exists  a (restriction) operator $R^\ep$ acting from $H^1_0(Q_\ep)^3$ into $H^1_0(\Omega_\ep)^3$ such that
\begin{enumerate}
\item $R^\ep \varphi=\varphi$, if $\varphi \in H^1_0(\Omega_\ep)^3$ (elements of $H^1_0(\Omega_\ep)$ are extended by $0$ to $Q_\ep$).
\item ${\rm div}R^\ep \varphi=0\hbox{  in }\Omega_\ep$, if ${\rm div}\varphi=0\hbox{  on }Q_\ep$.
\item For every $\varphi\in H^1_0(Q_\ep)^3$, there exists a positive constant $C$, independent of $\varphi$ and $\ep$, such that
\begin{equation}\label{estim_restricted}
\begin{array}{l}
\|R^\ep \varphi\|_{L^2(\Omega_\ep)^{3}}+ a_\ep\|D R^\ep \varphi\|_{L^2(\Omega_\ep)^{3\times 3}} \leq C\left(\|\varphi\|_{L^2(Q_\ep)}+a_\ep \|D \varphi\|_{L^2(Q_\ep)^{3\times 3}}\right)\,.
\end{array}
\end{equation}
\end{enumerate}
\end{lemma}
Using the restriction operator $R^\ep$ given in  Lemma \ref{restriction_operator}, we   introduce $F_\ep$ in $H^{-1}(Q_\ep)^3$ in the following way
\begin{equation}\label{F}\langle F_\varepsilon, \varphi\rangle_{H^{-1}(Q_\varepsilon)^3, H^1_0(Q_\ep)^3}=\langle \nabla p_\varepsilon, R^\varepsilon \varphi\rangle_{{H^{-1}(\Omega_\varepsilon)^3, H^1_0(\Omega_\ep)^3}}\,,\quad \hbox{for any }\varphi\in H^{1}_0(Q_\varepsilon)^3\,,
\end{equation}
and calcule the right hand side of (\ref{F}) by using (\ref{form_var_1}), which  gives
\begin{equation}\label{equality_duality}
\begin{array}{l}
\displaystyle
\left\langle F_{\varepsilon},\varphi\right\rangle_{H^{-1}(Q_\varepsilon)^3, H^1_0(Q_\ep)^3}=\displaystyle
-\nu\int_{\Omega_\varepsilon}D u_\ep : D R^{\varepsilon}\varphi\,dx+ \int_{\Omega_\varepsilon} f'\cdot (R^{\varepsilon}\varphi)'\,dx \,.
\end{array}\end{equation}
Using Lemma \ref{Estimates_lemma} for fixed $\ep$, we see that it is a bounded functional on $H^1_0(Q_\ep)$ (se the proof of Lemma \ref{Estimates_extended_lemma} below), and in fact $F_\ep\in H^{-1}(Q_\ep)^3$. Moreover, ${\rm div} \varphi=0$ implies $\left\langle F_{\varepsilon},\varphi\right\rangle=0\,,$ and the DeRham theorem gives the existence of $P_\varepsilon$ in $L^{2}_0(Q_\varepsilon)$ with $F_\varepsilon=\nabla P_\varepsilon$.

Next, we get for every $\tilde \varphi\in H^1_0(\Omega)^3$ where $\tilde \varphi(x', z_3)=\varphi(x',\ep z_3)$, using the change of variables (\ref{dilatacion}), that
$$\begin{array}{rl}\displaystyle\langle \nabla_{\varepsilon}\tilde P_\varepsilon, \tilde\varphi\rangle_{H^{-1}(\Omega)^3, H^1_0(\Omega)^3}&\displaystyle
=-\int_{\Omega}\tilde P_\varepsilon\,{\rm div}_{\varepsilon}\,\tilde\varphi\,dx'dy_3
=-\varepsilon^{-1}\int_{Q_\varepsilon}P_\varepsilon\,{\rm div}\,\varphi\,dx=\varepsilon^{-1}\langle \nabla P_\varepsilon, \varphi\rangle_{H^{-1}(Q_\varepsilon)^3, H^1_0(Q_\ep)^3}\,.
\end{array}$$
Using the identification (\ref{equality_duality}) of $F_\varepsilon$, we have
$$\begin{array}{l}\displaystyle\langle \nabla_{\varepsilon}\tilde P_\varepsilon, \tilde\varphi\rangle_{H^{-1}(\Omega)^3, H^1_0(\Omega)^3}\displaystyle
=\varepsilon^{-1}\Big(-\int_{\Omega_\varepsilon}D u_\ep : D R^{\varepsilon} \varphi\,dx
+\int_{\Omega_\varepsilon} f'(x')\cdot (R^{\varepsilon} \varphi)'\,dx\Big)\,,
\end{array}$$
and applying the change of variables (\ref{dilatacion}), we obtain 
\begin{equation}\label{extension_1}
\begin{array}{l}\displaystyle\langle \nabla_{\varepsilon}\tilde P_\varepsilon, \tilde\varphi\rangle_{H^{-1}(\Omega)^3, H^1_0(\Omega)^3}
=\displaystyle- \int_{\widetilde \Omega_\varepsilon}D_{\ep} \tilde u_\ep : D_{\ep} \tilde R^{\varepsilon} \tilde \varphi\,dx'dy_3
\displaystyle +\int_{\widetilde \Omega_\varepsilon} f'(x')\cdot (\tilde R^{\varepsilon} \tilde \varphi)'\,dx'dy_3\,,
\end{array}
\end{equation}
where $\tilde R^\ep\tilde \varphi=R^\ep \varphi$ for any $\tilde \varphi \in H^1_0(\Omega)^3$.

Finally, we estimate the right-hand side of (\ref{extension_1}) and give the estimate of the extended pressure $\tilde P_\ep$. 
\begin{lemma} \label{Estimates_extended_lemma} There exists a positive constant $C$ independent of $\ep$, such that 
\begin{equation}\label{esti_P}
\|\tilde P_\ep\|_{L^2(\Omega)}\leq C\,, \quad \|\nabla_\ep \tilde P_\ep\|_{H^{-1}(\Omega)^3}\leq C.
\end{equation}
\end{lemma}
\begin{proof}Applying the dilatation, we have that $\tilde R^\ep\tilde\varphi$ satisfies the following estimates
\begin{equation}\label{ext_1}\begin{array}{l}\displaystyle
\|\tilde R^\ep(\tilde\varphi)\|_{L^2(\widetilde\Omega_\ep)^3}\leq  C\left(\|\tilde\varphi\|_{L^2(\Omega)^3} 
+ a_\ep\|D_{x'}\tilde\varphi\|_{L^2(\Omega)^{3\times 2}}+{a_\ep\over \ep} \|\partial_{z_3}\tilde\varphi\|_{L^2(\Omega)^{3}}\right),

\\
\noame\displaystyle
\|D_{x'}\tilde R^\ep\tilde\varphi\|_{L^2(\widetilde\Omega_\ep)^{3\times 2}}\leq C\left({1\over a_\ep}\|\tilde\varphi\|_{L^2(\Omega)^3} 
+ \|D_{x'}\tilde\varphi\|_{L^2(\Omega)^{3\times 2}}+{1\over \ep} \|\partial_{z_3}\tilde\varphi\|_{L^2(\Omega)^{3}}\right),\\
\noame
\displaystyle
\|\partial_{y_3}\tilde R^\ep\tilde\varphi\|_{L^2(\widetilde\Omega_\ep)^{3}}\leq C\left({\ep\over a_\ep}\|\tilde\varphi\|_{L^2(\Omega)^3} 
+\ep \|D_{x'}\tilde\varphi\|_{L^2(\Omega)^{3\times 2}}+ \|\partial_{z_3}\tilde\varphi\|_{L^2(\Omega)^{3}}\right).
\end{array} 
\end{equation}
From the relation (\ref{parameters}), we have that 
\begin{equation}\label{ext_2}
\|\tilde R^\ep \tilde\varphi\|_{L^2(\widetilde\Omega_\ep)^3}\leq  C \|\tilde \varphi\|_{H^1_0(\Omega)^3},\quad \|D_\ep \tilde R^\ep\tilde \varphi\|_{L^2(\widetilde\Omega_\ep)^{3\times 3}}\leq {C\over a_\ep}\|\tilde \varphi\|_{H^1_0(\Omega)^3}.
\end{equation}

Thus,  from Cauchy-Schwarz's inequality and using estimates for $D_{\ep}\tilde u_\ep$ in (\ref{estimates_u_tilde}), assumption of $f'$ given in (\ref{fep}) and estimate of the dilated restricted operator (\ref{ext_2}),  we   obtain
$$
\begin{array}{l}
\displaystyle
\left|\int_{\widetilde\Omega_\ep}D_{\ep}\tilde u_\ep:D_{\ep}\tilde R^\ep\tilde\varphi\,dx'dy_3\right|\leq Ca_\ep\|D_{\ep}\tilde R^\ep\tilde\varphi\|_{L^2(\widetilde \Omega_\ep)^{3\times 3}}\leq C\|\tilde\varphi\|_{H^1_0(\Omega)^3},\\
\noame\displaystyle
 \left|\int_{\widetilde\Omega_\ep}f'\cdot (\tilde R^\ep \tilde\varphi)' \,dx'dy_3\right|\leq C\|\tilde R^\ep \tilde\varphi \|_{L^2(\widetilde\Omega_\ep)^3}\leq C\|\tilde\varphi\|_{H^1_0(\Omega)^3}\,,
\end{array}
$$
which together with (\ref{extension_1}) gives $$\left|\langle \nabla_{\varepsilon}\tilde P_\varepsilon, \tilde\varphi\rangle_{H^{-1}(\Omega)^3, H^1_0(\Omega)^3}\right|\leq C\|\tilde\varphi\|_{H^1_0(\Omega)^3}.$$
This implies 
$\|\nabla_{\ep}\tilde P_\ep\|_{L^2(\Omega)^3}\leq C$ and using the Ne${\breve{\rm c}}$as inequality, there exists a representative $\tilde P_\ep\in L^2_0(\Omega)$ such that
$$\|\tilde P_\ep\|_{L^2(\Omega)}\leq C\|\nabla\tilde P_\ep\|_{H^{-1}(\Omega)^3}\leq C\|\nabla_{\ep}\tilde P_\ep\|_{H^{-1}(\Omega)^3},$$
which implies (\ref{esti_P}).
\end{proof}
\subsection{Adaptation of the unfolding method}\label{sec:unfolding}
The change of variables (\ref{dilatacion}) does not provide the information we need about the behavior of of the solution in the microstructure associated to $\widetilde\Omega_\ep$. To solve this difficulty, we use an adaptation of the unfolding method (for classical versions see  \cite{arbogast, Ciora, Ciora2}) which is related with the change of variables applied in \cite{Bayada_coupling} to study the porous part in the case of modelling of a thin film passing
a thin porous media.  In a simple way, it consists of dividing the domain $\widetilde\Omega_\ep$ into cubes of lateral length $a_\ep$ and vertical length $a_\ep/\ep$. 
\begin{definition} Let $\tilde \varphi$ be in $L^2(\widetilde\Omega_\ep)$ and $\tilde \psi$ be in $L^2(\Omega)$. We define the functions $\hat \varphi_\ep\in L^2(\mathbb{R}^3\times Y_f)^3$ and $\hat \psi_\ep\in L^2(\mathbb{R}^3\times Y)^3$   by 
\begin{eqnarray}
\displaystyle\hat \varphi_\ep(x',z_3,y)= \tilde \varphi\left( a_{\varepsilon}\kappa\left(\frac{x^{\prime}}{a_{\varepsilon}}, { \ep z_3\over a_\ep } \right)e'+a_{\varepsilon}y^{\prime}, {a_{\varepsilon}\over \ep}\kappa\left(\frac{x^{\prime}}{a_{\varepsilon}}, { \ep z_3\over a_\ep   } \right)e_3+{a_{\varepsilon}\over \ep}y_3 \right),& \hbox{a.e. }(x^{\prime},z_3,y)\in  \mathbb{R}^3 \times   Y_f,&\label{def:unfolding1}\\
\noame
\displaystyle
\hat \psi_\ep(x',z_3,y)=  \tilde \psi\left( a_{\varepsilon}\kappa\left(\frac{x^{\prime}}{a_{\varepsilon}}, { \ep z_3\over a_\ep   } \right)e'+a_{\varepsilon}y^{\prime}, {a_{\varepsilon}\over \ep}\kappa\left(\frac{x^{\prime}}{a_{\varepsilon}}, { \ep z_3\over a_\ep   } \right)e_3+{a_{\varepsilon}\over \ep}y_3 \right),&  \hbox{a.e. }(x^{\prime},z_3,y)\in \mathbb{R}^3\times   Y,&\label{def:unfolding2}
\end{eqnarray}
assuming $\tilde \varphi$ (resp. $\tilde \psi$) is extended by zero outside $\widetilde\Omega_\ep$ (resp. $\Omega$), where
 the function $\kappa=(\kappa',\kappa_3)$ is defined by (\ref{kappa}).

\end{definition}
\begin{remark}\label{remarkCV}
The restrictions of $\hat \varphi_{\varepsilon}$  to $\widetilde Y_{k,a_{\varepsilon}}\times Y_f$  (resp.  $\hat \psi_\ep$ to $\widetilde Y_{k,a_{\varepsilon}}\times Y$) does not depend on $(x^{\prime},z_3)$, while as a function of $y$ it is obtained from $(\tilde \varphi, \tilde \psi)$ by using the change of variables 
\begin{equation}\label{CV}
y^{\prime}=\frac{x^{\prime}-a_{\varepsilon}k^{\prime}}{a_{\varepsilon}},\quad y_3=\frac{\ep z_3-a_{\varepsilon}k_3}{a_{\varepsilon}},
\end{equation}
which transforms $\widetilde Y_{f_k,a_{\varepsilon}}$ into $Y_f$ (resp. $\widetilde Y_{k,a_{\varepsilon}}$ into $Y$).
\end{remark}

\begin{proposition} \label{properties_um}Let $\tilde \varphi$ be in $L^2(\widetilde\Omega_\ep)$ and $\tilde \psi$ be in $L^2(\Omega)$. Then, we have
\begin{equation*}\label{relation_norms}
\begin{array}{lll}
\displaystyle
\|\hat \varphi_\ep\|_{L^2(\mathbb{R}^3\times Y_f)^3}=\|\tilde \varphi\|_{L^2(\widetilde \Omega_\ep)^3},& \|D_{y'}\hat \varphi_\ep\|_{L^2(\mathbb{R}^3\times Y_f)^{3\times 2}}=a_\ep\|D_{x'}\tilde \varphi\|_{L^2(\widetilde \Omega_\ep)^{3\times 2}},&\displaystyle \|\partial_{y_3}\hat \varphi_\ep\|_{L^2(\mathbb{R}^3\times Y_f)^3}={a_\ep\over \ep}\|\partial_{z_3}\tilde \varphi\|_{L^2(\widetilde \Omega_\ep)^3},\\
\noame
\|\hat \psi_\ep\|_{L^2(\mathbb{R}^3\times Y)^3}=\|\tilde \psi\|_{L^2(\Omega)^3}
,& 
\|D_{y'}\hat \psi_\ep\|_{L^2(\mathbb{R}^3\times Y)^{3\times 2}}=a_\ep\|D_{x'}\tilde \psi\|_{L^2(\Omega)^{3\times 2}}
,&\displaystyle 
\|\partial_{y_3}\hat \psi_\ep\|_{L^2(\mathbb{R}^3\times Y)^3}={a_\ep\over \ep}\|\partial_{z_3}\tilde \psi\|_{L^2(\Omega)^3}.
\end{array}\end{equation*}
\end{proposition}
\begin{proof} We will only make the proof for $\hat \varphi_\ep$. The procedure for $\hat \psi$ is similar, so we omit it. Taking into account the definition (\ref{def:unfolding1}) of $\hat{\varphi}_{\varepsilon}$, we obtain
\begin{eqnarray*}
\int_{\mathbb{R}^3\times Y_f}\left\vert D_{y^{\prime}} \hat{\varphi}_{\varepsilon}(x^{\prime},z_3,y) \right\vert^2 dx^{\prime}dz_3dy&=&\displaystyle\sum_{k\in  \mathbb{R}^3}\int_{\widetilde Y_{k,a_{\varepsilon}}}\int_{Y_f}\left\vert D_{y^{\prime}} \hat{\varphi}_{\varepsilon}(x^{\prime},z_3,y) \right\vert^2 dx^{\prime}dz_3dy\\
&=&\displaystyle\sum_{k \in \mathbb{R}^3}\int_{\widetilde Y_{k,a_{\varepsilon}}}\int_{Y_f}\left\vert D_{y^{\prime}} \tilde{\varphi} (a_{\varepsilon}k^{\prime}+a_{\varepsilon}y^{\prime},a_{\varepsilon}\ep^{-1}k_3+a_{\varepsilon}\ep^{-1}y_3
) \right\vert^2dx^{\prime}dz_3dy.
\end{eqnarray*}
We observe that $\tilde{u}_{\varepsilon}$ does not depend on $(x^{\prime},z_3)$, then we can deduce
\begin{eqnarray*}
\int_{\mathbb{R}^3\times Y_f}\left\vert D_{y^{\prime}} \hat{\varphi}_{\varepsilon}(x^{\prime},z_3,y) \right\vert^2 dx^{\prime}dz_3dy
= {a_{\varepsilon}^3\over \ep}\displaystyle\sum_{k \in \mathbb{R}^3}\int_{Y_f}\left\vert D_{y^{\prime}} \tilde{\varphi}(a_{\varepsilon}k^{\prime}+a_{\varepsilon}y^{\prime},a_{\varepsilon}\ep^{-1}k_3+a_{\varepsilon}\ep^{-1}y_3
) \right\vert^2dy.
\end{eqnarray*}
By the change of variables (\ref{CV}), we obtain
\begin{eqnarray*}
\int_{\mathbb{R}^3\times Y_f}\left\vert D_{y^{\prime}} \hat{\varphi}_{\varepsilon}(x^{\prime},z_3,y) \right\vert^2 dx^{\prime}dz_3dy
&=&a_{\varepsilon}^2
\displaystyle\sum_{k \in  \mathbb{R}^3}\int_{\widetilde Y_{f_k,a_{\varepsilon}}} \left\vert D_{x^{\prime}} \tilde{\varphi}(x^{\prime},z_3) \right\vert^2dx^{\prime}dz_3\\
&=& a_{\varepsilon}^2\int_{\widetilde \Omega_\ep}\left\vert D_{x^{\prime}} \tilde \varphi(x^{\prime},z_3) \right\vert^2dx^{\prime}dz_3.
\end{eqnarray*}
Thus, we get the property for $D_{y'}\hat \varphi_\ep$.

Similarly,  we have
\begin{eqnarray*}
\int_{\mathbb{R}^3\times Y_f}\left\vert \partial_{y_3} \hat{\varphi}_{\varepsilon}(x^{\prime},z_3,y) \right\vert^2dx^{\prime}dz_3dy= {a_{\varepsilon}^3\over \ep}\displaystyle\sum_{k \in \mathbb{R}^3}\int_{Y_f}\left\vert \partial_{y_3} \tilde{\varphi} (a_{\varepsilon}k^{\prime}+a_{\varepsilon}y^{\prime},a_{\varepsilon}\ep^{-1}k_3+a_{\varepsilon}\ep^{-1}y_3
)\right\vert^2dy.
\end{eqnarray*}
By the change of variables (\ref{CV})  we obtain
\begin{eqnarray*}
\int_{\mathbb{R}^3\times Y_f}\left\vert \partial_{y_3} \hat{\varphi}_{\varepsilon}(x^{\prime},z_3,y) \right\vert^2dx^{\prime}dz_3dy 
& 
=
&
{a_{\varepsilon}^2\over \ep^2}\displaystyle\sum_{k \in \mathbb{R}^3}\int_{\widetilde Y_{f_k,a_\ep}}\left\vert \partial_{y_3} \tilde{\varphi} (x',z_3)\right\vert^2dx'dz_3.
\\
\noame
\displaystyle
&=&{a_{\varepsilon}^2\over \ep^2} \int_{\widetilde\Omega_\ep}\left\vert \partial_{y_3} \tilde{\varphi}(x^{\prime},z_3)\right\vert^2dx^{\prime}dz_3,
\end{eqnarray*}
so the the property for $\partial_{y_3}\hat \varphi_\ep$ is proved. Finally, reasoning analogously we deduce
\begin{eqnarray*}
\int_{\mathbb{R}^3\times Y_f}\left\vert \hat{\varphi}_{\varepsilon}(x^{\prime},z_3,y)\right\vert^2dx^{\prime}dz_3dy = \int_{\widetilde\Omega_\ep}\left\vert \tilde{\varphi}(x^{\prime},z_3)\right\vert^2dx^{\prime}dz_3,
\end{eqnarray*}
and the property for $\hat \varphi_\ep$  holds. 
\end{proof}
Now, from functions $\tilde u_\ep$ and $\tilde P_\ep$  we define $\hat u_\ep$ by using (\ref{def:unfolding1}) and $\hat P_\ep$ by means of (\ref{def:unfolding2}). Below, we get the estimates for these sequences. 

\begin{lemma}\label{estimates_hat}
There exists a constant $C>0$ independent of $\ep$, such that $\hat u_\ep$ defined by (\ref{def:unfolding1}) and   $\hat P_\ep$ defined by (\ref{def:unfolding2}) satisfy  
\begin{equation}\label{estim_u_hat}
 \|\hat u_\ep\|_{L^2(\mathbb{R}^3\times Y_f)^3}\leq Ca_\ep^2,\quad 
 \|D_{y}\hat u_\ep\|_{L^2(\mathbb{R}^3\times Y_f)^{3\times 3}}\leq Ca_\ep^2,
 \end{equation}
\begin{equation}\label{estim_P_hat}
 \|\hat P_\ep\|_{L^2(\mathbb{R}^3\times Y)}\leq C.
\end{equation}
 \end{lemma}
\begin{proof} Estimates (\ref{estim_u_hat}) and (\ref{estim_P_hat}) easily follow from Proposition \ref{properties_um} and estimates given in Lemmas \ref{Estimates_lemma} and \ref{Estimates_extended_lemma}.
\end{proof}
 To finish this section, we will give the variational formulation satisfied by the functions $(\hat u_\ep,\hat P_\ep)$, which will be useful in the following sections. Thus, we consider $\varphi_\ep(x',z_3)=\varphi(x',z_3,x'/a_\ep,\ep z_3/a_\ep)$ as test function in (\ref{form_var_2}) where $\varphi(x',z_3,y)\in C_c^1(\Omega;H_{{\rm per}}^1(Y)^3)$.  Taking into account the extension of the pressure 
$$\int_{\widetilde\Omega_\ep}\nabla_{\ep}\tilde p_\ep\cdot \varphi_\ep\,dx'dz_3=\int_{\Omega}\nabla_{\ep}\tilde P_\ep\cdot \varphi_\ep,\,dx'dz_3\,,$$
the variational formulation (\ref{form_var_2}) reads
\begin{equation}\label{form_var_general_1}
\begin{array}{l}\displaystyle
\int_{\widetilde\Omega_\ep}D_{\ep}\tilde u_\ep:D_{\ep}\varphi_\ep\,dx'dy_3-\int_{\Omega}
\tilde P_\ep\,{\rm div}_{\ep}\varphi_\ep\,dx'dy_3
=\int_{\widetilde\Omega_\ep}f'\cdot \varphi_\ep'\,dx'dy_3\,.
\end{array}
\end{equation}
By the change of variables given in Remark \ref{remarkCV}, we obtain
\begin{equation}\label{form_var_hat_u}
\begin{array}{l}
\displaystyle{\nu\over a_\ep^2}\int_{\Omega\times Y_f}D_{y}\hat u_\ep:D_{y}\varphi\,dx'dz_3dy
-\int_{\Omega\times Y}\hat P_\ep\,{\rm div}_{\ep}\varphi\,dx'dz_3dy-{1\over a_\ep}\int_{\Omega\times Y}\hat P_\ep\,{\rm div}_{y}\varphi\,dx'dz_3dy\\
\noame
\displaystyle
=
\int_{\Omega\times Y_f}f'\cdot \varphi'\,dx'dz_3dy+O_\ep\,,
\end{array}
\end{equation}
When $\ep$ tends to zero, we will analyze the asymptotic behavior  of sequence $(\hat u_\ep,  \hat P_\ep)$ in the next sections.
\subsection{Some compactness results}\label{sec:compactness}
In this subsection we obtain some compactness results concerning the behavior of the sequences $(\tilde U_\ep, \tilde P_\ep)$ and $(\hat u_\ep,\hat p_\ep)$.
\begin{lemma}\label{lemma_compactness}
For a subsequence of $\ep$ still denoted by $\ep$,  there exist   $u\in L^2(\Omega)^3$ and $\hat u\in L^2(\mathbb{R}^3; H^1_{{\rm per}}(Y_f)^3)$, such that
\begin{eqnarray}
&a_\ep^{-2}\tilde U_\ep \rightharpoonup (u',0)\hbox{ in }L^2(\Omega)^3,&\label{conv_vel_tilde}\\
\noame
&a_\ep^{-2}\hat u_\ep\rightharpoonup \hat u\hbox{ in }L^2(\mathbb{R}^3; H^1(Y_f)^3),\quad a_\ep^{-2}D_y\hat u_\ep\rightharpoonup D_y\hat u\hbox{ in }L^2(\mathbb{R}^3\times Y_f)^3,\label{conv_vel_gorro}&
\end{eqnarray}
with the boundary conditions $u=0$ on $z_3=\{0,1\}$ and $\hat u=0$ in $\Omega\times  T$ and in  $(\mathbb{R}^3\setminus \Omega)\times Y_f$.

Moreover, defining $U(x')=\int_0^1u(x',z_3)\,dz_3$ and  $\hat U(x',y)=\int_0^1\hat U(x',z_3,y)\,dz_3$,  it holds 
 \begin{equation}\label{relation_u_ugorro}
 U(x')= \int_{Y_f}\hat U(x',y)\,dy\quad\hbox{with}\quad \int_{Y_f}\hat U_3(x',y)\,dy=0\quad \hbox{ in  }\omega,
 \end{equation}
and  $\hat U(x',y)\in L^2(\mathbb{R}^2; H^1_{\rm per}(Y_f)^3)$  satisfies $\hat U=0$ in $\mathbb{R}^2\times T$ and in $(\mathbb{R}^2\setminus\omega)\times Y_f$ together with divergence conditions 
\begin{eqnarray}
& \displaystyle {\rm div}_{y}\,\hat U(x',y)=0\quad \hbox{in }\mathbb{R}^3\times Y_f,&\label{divyproperty}\\
\noame
&\displaystyle {\rm div}_{x'}\left(\int_{Y_f}\hat U'(x',y)\,dy\right)=0\quad \hbox{in }\mathbb{R}^2,\label{divyproperty_hat}\\
\noame
&\displaystyle \left(\int_{Y_f}\hat U'(x',y)\,dy\right)\cdot n=0\quad \hbox{on }\partial\omega,& \label{divyproperty_hat2}
\end{eqnarray}
\end{lemma}
\begin{proof} We start with the extended velocity $\tilde U_\ep$. From estimates (\ref{estimates_u_tilde}), we get
$$\|\tilde U_\ep\|_{L^2(\Omega)^3}\leq C a_\ep^2,\quad \|D_{x'}\tilde U_\ep\|_{L^2(Q_\ep)^3}\leq Ca_\ep ,\quad \|\partial_{y_3}\tilde U_\ep\|_{L^2(Q_\ep)^3}\leq C a_\ep\ep.$$

Then, applying Lema 3.12-$(ii)$ in \cite{Anguiano_MJOM}, we get the existence of $u\in L^2(\Omega)^3$ with $u_3=0$ on $z_3=\{0,1\}$, such that it holds convergence (\ref{conv_vel_tilde}), up to a subsequence, and the following divergence condition 
\begin{equation}\label{divxproperty}
{\rm div}_{x'} U'(x')=0\quad \hbox{in }\omega,\quad U'(x')\cdot n=0\quad \hbox{in }\partial\omega.
\end{equation}

Next,  taking into account that from estimates of the velocity $\hat u_\ep$ given in (\ref{estim_u_hat}) we have the existence of $\hat u\in L^2(\mathbb{R}^3; H^1_{\rm per}(Y_f)^3)$ satisfying, up to a subsequence,  convergences (\ref{conv_vel_gorro}). Taking into account that $a_\ep^{-2}\hat u_\ep$ vanishes on $\widetilde \Omega_\ep\times T$, we deduce that $\hat u$ also vanishes on  $\Omega\times T$.  Moreover, by construction $\hat u_\ep$ is zero outside $\widetilde\Omega_\ep$ and so $\hat u$ vanishes on $(\mathbb{R}^3\setminus \Omega)\times Y_f$. Since ${\rm div}_\ep \tilde u_\ep=0$ in $\widetilde \Omega_\ep$,  by applying the change of variables (\ref{CV}) we get 
$$a_\ep^{-1}{\rm div}_y \hat u_\ep=0. $$
Multiplying by $a_\ep^{-1}$ and passing to the limit by using convergence (\ref{conv_vel_gorro}), we deduce ${\rm div}_y\,\hat u=0$ in $\mathbb{R}^3\times Y_f$ and so we get (\ref{divyproperty}).

It remains to prove that $\hat u$ is periodic in $y$. Thus follows by passing to the limit in the equality
$$a_\ep^{-2}\hat u_\ep\left(x'+a_\ep e_1, z_3,-{1\over 2}, y_2,y_3\right)=a_\ep^{-2}\hat u_\ep\left(x',z_3,{1\over 2}, y_2,y_3\right),$$
which is a consequence of definition (\ref{def:unfolding1}). This shows 
$$\hat u\left(x',z_3,-{1\over 2},y_2,y_3\right)=\hat u\left(x',z_3,{1\over 2},y_2,y_3\right),$$
and then is proved  the periodicity of $\hat u$ with respect to $y_1$. Similarly, we prove the periodicity with respect to $y_2$. To prove the periodicity with respect to $y_3$, we consider 
$$a_\ep^{-2}\hat u_\ep\left(x', z_3+{a_\ep\over \ep},y_1, y_2,-{1\over 2}\right)=a_\ep^{-2}\hat u_\ep\left(x',z_3,y_1, y_2,{1\over 2}\right),$$
and passing to the limit we have 
$$\hat u_\ep\left(x', z_3,y_1, y_2,-{1\over 2}\right)=a_\ep^{-2}\hat u_\ep\left(x',z_3,y_1, y_2,{1\over 2}\right),$$
which shows the periodicity with respect to $y_3$.

To finish, we prove relation (\ref{relation_u_ugorro}). To do this, from  the change of variables (\ref{CV}), we have
$$\int_{\widetilde \Omega_\ep} \tilde u_\ep(x',z_3)\,dx'dz_3=\int_{\Omega\times Y_f}\hat u_\ep(x',z_3,y)\,dx'dz_3dy.$$
By using extension of the velocity and the dilatation (\ref{dilatacion}), we get
$$\int_{\Omega} \tilde {U}_\ep(x',z_3)\,dx'dz_3=\int_{\Omega\times Y_f}\hat u_\ep(x',z_3,y)\,dx'dz_3dy.$$
Multiplying this equality by $a_\ep^{-2}$, passing to the limit by using convergences (\ref{conv_vel_gorro}) and (\ref{conv_vel_tilde}), we deduce
$$\int_\Omega u(x',z_3)\,dz_3=\int_{\Omega\times Y_f}\hat u(x',z_3,y)\,dy.$$
which, taking into account that $u_3=0$ in $\Omega$,  implies relation (\ref{relation_u_ugorro}). Finally, this together with relation (\ref{divxproperty}) gives the divergence condition (\ref{divyproperty_hat}). Finally, condition $\hat u=0$ in $(\mathbb{R}^3\setminus \Omega)\times Y_f$ implies  $\hat U=0$ in $(\mathbb{R}^2\setminus \omega)\times Y_f$. This and (\ref{divyproperty_hat}) imply (\ref{divyproperty_hat2}).
\end{proof}

\begin{lemma}\label{lemma_conv_pressure}
For a subsequence of $\ep$ still denoted by $\ep$, there exists  $p\in L^2_0(\omega)$ independent of $y_3$, such that 
 \begin{equation}\label{conv_pressure_sub}
\tilde P_\ep\to p\quad\hbox{ in }L^2(\Omega),
\end{equation} 
 \begin{equation}\label{conv_pressure_gorro}
\hat p_\ep \to p \quad\hbox{ in }L^2(\mathbb{R}^2\times Y)^2.
\end{equation}

\end{lemma}
\begin{proof} Taking into account estimate the first estimate in (\ref{esti_P}) and (\ref{estim_P_hat}), we deduce that there exist $p\in L^2(\Omega)$ and $\hat p\in L^2(\Omega\times Y)$  such that, up to a subsequence, 
\begin{equation}\label{conv_pressure_sub_weak}
\tilde P_\ep\rightharpoonup p\quad\hbox{ in }L^2(\Omega),
\end{equation} 
 \begin{equation}\label{conv_pressure_gorro_weak}
\hat p_\ep \rightharpoonup p \quad\hbox{ in }L^2(\mathbb{R}^2\times Y).
\end{equation}

From convergence (\ref{conv_pressure_sub_weak}) we deduce that $\partial_{y_3}\tilde P_\ep$ also converges to $\partial_{y_3}p$ in $H^{-1}(\Omega)$. Also, from the second estimate in (\ref{esti_P}), we can deduce that  by noting that $\partial_{z_3}\tilde P_\ep$ converges to zero in  in  $H^{-1}(\Omega)$.  Then, by the uniqueness of the limit  we obtain $\partial_{z_3} p=0$ and so $\tilde p$ is independent of $z_3$. Since $\tilde p_\ep$ has null mean value in $\widetilde \Omega_\ep$, then $p$ has null mean value in $\omega$. Moreover, reasoning similarly to obtain relation (\ref{relation_u_ugorro}), it follows the relation $\tilde p(x')=\int_Y \hat p(x',z_3,y)\,dy$.

Finally, we shall prove that $\hat p$ is in fact equal to $p$ in $\omega$. To do that, it remains to prove that $\hat p$ does not depend on $y$. To do this  we multiply by $a_\ep$ the variational formulation (\ref{form_var_hat_u}), which gives 
$$
\begin{array}{l}\displaystyle{\nu\over a_\ep}\int_{\Omega\times Y_f}D_{y}\hat u_\ep:D_{y}\varphi\,dx'dz_3dy
-a_\ep\int_{\Omega\times Y}\hat P_\ep\,{\rm div}_{x'}\varphi'\,dx'dz_3dy-{a_\ep\over \ep}\int_{\Omega\times Y}\hat P_\ep\,\partial_{z_3}\varphi_3\,dx'dz_3dy\\
\noame
\displaystyle-\int_{\Omega\times Y}\hat P_\ep\,{\rm div}_{y}\varphi\,dx'dz_3dy=
a_\ep\int_{\Omega\times Y_f}f'\cdot \varphi'\,dx'dz_3dy+O_\ep\,,
\end{array}$$
 Thus, by using relation (\ref{parameters}) and convergences (\ref{conv_vel_gorro}) and (\ref{conv_pressure_sub}) and (\ref{conv_pressure_sub}), we  pass  to the limit  when $\ep$ tends to zero and we get 
$$\int_{\Omega\times Y}\hat p\,{\rm div}_{y}\varphi\,dx'dz_3dy=0, \quad \forall \varphi\in C_c^1(\Omega;H_{{\rm per}}^1(Y)^3),$$
 By density  it holds in $L^2(\Omega;H^1_{\rm per}(Y)^3)$ and implies that $p$ does not depend on the variable  $y$.

Finally, following  \cite{Tartar} adapted to the case of thin domains, we prove that the convergence of the pressure is in fact strong. As $\tilde u_3=0$ and $\tilde p$ only depends on $x'$, let $\sigma_\ep(x',z_3)=(\sigma_\ep'(x'),0)\in H^1_0(\omega)^3$ be such that 
\begin{equation}\label{strong_p_1}
\sigma_\ep\rightharpoonup \sigma\quad\hbox{in }H^1_0(\omega)^3. 
\end{equation}
Then, we have
$$\begin{array}{rl}
\displaystyle
\left|<\nabla_{\ep}\tilde P_\ep,\sigma_\ep>_{H^{-1}(\Omega),H^1_0(\Omega)}-<\nabla_{x'}p,\sigma>_{H^{-1}(\Omega),H^1_0(\Omega)}\right| \leq &\displaystyle
\left|<\nabla_{\ep}\tilde P_\ep,\sigma_\ep-\sigma>_{H^{-1}(\Omega),H^1_0(\Omega)}\right|\\
\noame\displaystyle &+\left|<\nabla_{\ep}\tilde P_\ep-\nabla_{x'}p,\sigma>_{H^{-1}(\Omega),H^1_0(\Omega)}\right|.
\end{array}$$
On the one hand, using convergence (\ref{lemma_conv_pressure}), we have 
$$\left|<\nabla_{\ep} \tilde P_\ep-\nabla_{x'}p,\sigma>_\Omega\right|=\left|\int_\Omega\left(\tilde P_\ep-p\right)\,{\rm div}_{x'}\sigma'\,dx\right|\to 0,\quad \hbox{as }\ep\to 0\,.$$
On the other hand, from (\ref{extension_1}) and proceeding as in the proof of Lemma \ref{Estimates_extended_lemma}, we have
$$\begin{array}{ll}
\left|<\nabla_{\ep}\tilde P_\ep,\sigma_\ep-\sigma>_\Omega\right|=& \left|<\nabla_{x'}\tilde P_\ep,\tilde R^\ep(\sigma_\ep'-\sigma')>_{\widetilde\Omega_\ep}\right|\\
\noame
& \leq C\left(\|\sigma'_\ep-\sigma'\|_{L^2(\omega)^3}+a_\ep\|D_{x'} (\sigma'_\ep-\sigma')\|_{L^2(\omega)^{3\times 2}}\right)\to 0\quad\hbox{as }\ep\to 0,
\end{array}$$
by virtue of (\ref{strong_p_1}) and the Rellich theorem. This implies that $\nabla_{\ep}\tilde P_\ep\to \nabla_{x'} p$ strongly in $H^{-1}(\Omega)^3$, which together the Ne${\breve{\rm c}}$as inequality, implies the strong convergence of the pressure $\tilde P_\ep$ given in (\ref{conv_pressure_sub}).  This and and the first relation given in (\ref{relation_norms}) prove that $\hat P_\ep$ also converges strongly to $p$. 
\end{proof}

\subsection{Proof of main result}\label{sec:mainthm}
In this subsection, we describe the asymptotic behavior of solution of problem (\ref{system_1_dimension_1}) by using the convergences given in Lemmas \ref{lemma_compactness} and \ref{lemma_conv_pressure}. First, we give the homogenized system satisfied by the unfolding functions and then, we introduce the local problems and give the homogenized flow model.
\begin{proof}[Proof of Theorem \ref{mainthm}]
We divide the proof in two steps. 

{\it Step 1}. From  Lemmas \ref{lemma_compactness} and \ref{lemma_conv_pressure}, we have that there exist $\hat u\in L^2(\mathbb{R}^2;H^1_{\rm per}(Y_f)^3)$ and $p\in L^2_0(\omega)$ as the limits of $a_\ep^{-2}\hat u_\ep$ and $\hat P_\ep$ satisfying convergences (\ref{conv_vel_gorro}) and (\ref{conv_pressure_gorro}) respectively.  In this step, we prove that  $(\hat U,p)$ satisfies the following homogenized system
\begin{equation}\label{system_1_2_hat}\left\{\begin{array}{rl}
-\nu\Delta_z \hat U+ \nabla_z\hat q= f' -\nabla_{x'} p & \hbox{in}\quad\omega\times Y_f, \\
\noame\displaystyle
{\rm div}_y \hat U =0& \hbox{in}\quad\omega\times Y_f, \\
\noame\displaystyle
{\rm div}_{x'}\left(\int_{Y_f}\hat U'\,dy\right)=0&\hbox{in }\omega, \\
\noame\displaystyle
\left(\int_{Y_f}\hat U'\,dy\right)\cdot n=0&\hbox{on }\partial\omega, \\
\noame\displaystyle
\hat U=0& \hbox{in }\omega\times T,\\
\noame
\hat q\in  L^2(\omega;L^2_{\rm per}(Y)/\mathbb{R}).
\end{array}\right.
\end{equation}
Taking into account Lemma \ref{lemma_compactness}, we have that $\hat U$ satisfies   conditions (\ref{system_1_2_hat})$_{2,3,4,5}$.

Now, we want to prove that $(\hat U, p)$ satisfies  equation (\ref{system_1_2_hat})$_1$. To do this, we consider $\varphi_\ep(x',z_3)=(\varphi'_\ep,\varphi_{\ep,3})$ with $\varphi_\ep'=\varphi'(x',z_3,x'/a_\ep,\ep z_3/a_\ep)$ and $\varphi_{\ep,3}=\ep \varphi_3(x',z_3,x'/a_\ep,\ep z_3/a_\ep)$ as test function in (\ref{form_var_hat_u}) where $\varphi(x',z_3,y)\in C_c^1(\Omega;H_{{\rm per}}^1(Y)^3)$  such that $\varphi=0$ in $\Omega\times T$,  ${\rm div}_{x'}(\int_{Y_f}\varphi'\,dy)=0$ in $\Omega$ and ${\rm div}_y \varphi=0$ in $\Omega\times Y_f$. Then, the variational formulation reads
$$
\begin{array}{l}\displaystyle {\nu\over a_\ep^2}\int_{\Omega\times Y_f}D_{y}\hat u_\ep :D_{y}\varphi\,dx'dz_3dy
- \int_{\Omega\times Y} \hat P_\ep\,{\rm div}_{x'}\varphi'\,dx'dz_3dy- \int_{\Omega\times Y}\hat P_\ep\,\partial_{z_3}\varphi_3\,dx'dz_3dy\\
\noame
\displaystyle=
\int_{\Omega\times Y_f}f'\cdot \varphi'\,dx'dz_3dy+O_\ep\,,
\end{array}$$
 We pass to the limit by using convergences (\ref{conv_vel_gorro}) and (\ref{conv_pressure_gorro}) and we obtain 
$$
\begin{array}{l}\displaystyle  \nu\int_{\Omega\times Y_f}D_{y}\hat u :D_{y}\varphi\,dx'dz_3dy
- \int_{\Omega\times Y}p\,{\rm div}_{x'}\varphi'\,dx'dz_3dy- \int_{\Omega\times Y}p\,\partial_{z_3}\varphi_3\,dx'dz_3dy=
\int_{\Omega\times Y_f}f'\cdot \varphi'\,dx'dz_3dy\,.
\end{array}$$
Taking into account that $p$ does not depend on $z_3$, we have that 
$$\int_{\Omega\times Y}p\,{\rm div}_{x'}\varphi'\,dx'dz_3dy=\int_{\Omega}p\,\left({\rm div}_{x'}\int_{Y_f}\varphi'\,dy\right)\,dx'dz_3=0\quad\hbox{and}\quad \int_{\Omega\times Y}p\,\partial_{z_3}\varphi_3\,dx'dz_3dy=0.$$
Then we obtain
\begin{equation}\label{limit_form_var}
\begin{array}{l}\displaystyle \nu \int_{\Omega\times Y_f}D_{y}\hat u :D_{y}\varphi\,dx'dz_3dy=
\int_{\Omega\times Y_f}f'\cdot \varphi'\,dx'dz_3dy\,.
\end{array}
\end{equation}
We take into account that there is no $z_3$-dependence in the obtained variational formulation. For that, we can
consider $\varphi$ independent of $z_3$, which implies that $(\hat U, p)$ satisfies the same variational formulation with integrals
in $\omega\times Y_f$. By density, this equality holds for every function in the Hilbert space $V$  defined by 
$$
V=\left\{\begin{array}{l}
\varphi(x',y)\in L^2(\omega;H^1_{\rm per}(Y)^3) \hbox{ such that }\\
\noame
\displaystyle {\rm div}_{x'}\left(\int_{Y_f}\varphi(x',y)\,dy \right)=0\hbox{ in }\omega,\quad \left(\int_{Y_f}\varphi(x',y)\,dy \right)\cdot n=0\hbox{ on }\partial\omega\\\noame
{\rm div}_y\varphi(x',y)=0\hbox{ in }\omega\times Y_f,\quad \varphi(x',y)=0\hbox{ in }\omega\times T
\end{array}\right\}\,.
$$
By Lax-Milgram lemma, the variational formulation (\ref{limit_form_var}) in the Hilbert $V$ admits a unique solution $\hat u$ in $V$. Reasoning as in  \cite{Allaire0}, the orthogonal of $V$ with respect to the usual scalar product in $L^2(\omega\times Y)$ is made of gradients of the form 
 $\nabla_{x'}q(x')+\nabla_z \hat q(x',z)$, with $q(x')\in L^2(\omega)/\mathbb{R}$ and $\hat q(x',y)\in L^2(\omega;L^2_{\rm per}(Y)/\mathbb{R})$. 
 Therefore, since $\hat U\in L^2(\omega;H^1_{\rm per}(Y_f)^3)$, $\hat U=0$ in $\omega\times T$, ${\rm div}_y\hat U=0$ in $\omega\times Y_f$, by integration by parts  we deduce that $\hat U$, $p$ and $\hat q$ satisfy system (\ref{system_1_2_hat}). It remains to prove that  $q$ coincides with pressure $p$. This can be easily done
by multiplying the variational formulation (\ref{form_var_hat_u}) by a test function $\varphi$ independent of $z_3$ with ${\rm div}_y$ equal to zero, and identifying limits. 

{\it Step 2}. In this step we deduce the expression for velocity $U$ given in (\ref{Darcy_velocity}) and the Darcy equation satisfied by $p$ given in (\ref{Darcy}). To do this, let us define the local problems which are useful to eliminate the variable $y$ of the previous homogenized
problem and then obtain a Darcy equation for the pressure $p$.  

Thus, we consider  $(w^{i},\pi^i)$, $i=1,2,3$, as the unique solution of problem (\ref{Local_problems}), see \cite{Tartar}, where we observe that 
\begin{equation}
w^3=0\quad\hbox{in }H^1_{\rm per}(Y_f)^3,\quad \pi^3=y_3\quad\hbox{in }L^2_{\rm per}(Y_f)/\mathbb{R}.
\end{equation}
Then, reasoning by linearity and uniqueness, we deduce 
\begin{equation}\label{identification}
\hat U(x',y)=\sum_{i=1}^2\left(f_i'(x')-\partial_{x_i}p(x')\right)w^i(y),\quad \hat q(x',y)=\sum_{i=1}^2\left(f_i'(x')-\partial_{x_i}p(x')\right)\pi^i(y).
\end{equation}
From relation (\ref{relation_u_ugorro}), we deduce expression (\ref{Darcy_velocity}) where the matrix $K$ is given by 
$$K_{ij}=\int_{Y_f}w^i_j(y)\,dy\quad i,j=1,2.$$
Let us now prove the Darcy equation (\ref{Darcy}). To do this, we recall that $\hat U$ satisfies conditions (\ref{system_1_2_hat})$_{3,4}$ and therefore, using identification (\ref{identification}), we get
\begin{equation}\label{identi_darcy}
\begin{array}{rl}
\displaystyle\sum_{j=1}^2\partial_{x_j}\left(\sum_{i=1}^2(f_i'-\partial_{x_i}p)\int_{Y_f}w^i_j\,dy\right)={\rm div}_{x'}\int_{Y_f}\hat U'(x',y)\,dy=0&\quad\hbox{in }\omega,\\
\noame
\displaystyle
\left(\sum_{i=1}^2(f_i'-\partial_{x_i}p)\int_{Y_f}w^i_j\,dy\right)\cdot n=\left(\int_{Y_f}\hat U'(x',y)\,dy\right)\cdot n=0&\quad\hbox{on }\partial \omega.
\end{array}
\end{equation}
On the other hand, we observe that taking $w^i$ as test function in the equation satisfied by $w^j$, we have 
\begin{equation}\label{Atest}
\int_{Y_f}w^i_j(y)\,dy=\nu\int_{Y_f}Dw^{i}:Dw^j\,dy,\quad i,j=1,2,3.
\end{equation}
Thus,  (\ref{identi_darcy}) proves that $p$ satisfies the Darcy equation (\ref{Darcy}), which is an elliptic equation with  $K$  a symmetric and positive definite matrix, see \cite{Tartar}.  Then problem (\ref{Darcy}) has a unique solution. Moreover, as  $K f'\in L^2(\Omega)$ and $Y_f$ is a subset of $Y$ which is smooth and connected, from regularity results for problem  (see Chapter 7 in \cite{Sanchez}) we conclude that  $p\in H^1(\omega)$, which implies that $U$ (and so $\hat U$) is also unique. This proves that the entire sequence $(a_\ep^{-2}\tilde u_\ep,p_\ep)$ converges to $(u,p)$.
 This finishes the proof.
 \end{proof}

\end{document}